\documentclass[11 pt]{amsart}
\usepackage{mathrsfs}     
\usepackage{amsthm}
\usepackage{amssymb}
\usepackage{latexsym}
\usepackage{amsfonts}
\usepackage{amsmath}
\usepackage{amstext}
\usepackage{amsrefs}
\newtheorem{theorem}{Theorem}[section]
\newtheorem{lemma}[theorem]{Lemma}
\newtheorem{corollary}[theorem]{Corollary}

\newtheorem{definition}[theorem]{Definition}
\newtheorem{claim}[theorem]{Claim}

\begin{document}

\title{An Interpolating Curvature Condition Preserved By Ricci Flow}

\author{Xiang Gao}


\address{Department of Mathematics,~East China Normal University,
~Lane 500, DongChuan Road,~Minhang District,~Shanghai
City,~200241,~People's Republic of China}

\address{Department of Mathematics,~Ocean University of China,
~Lane 238, SongLing Road,~Laoshan District,~Qingdao
City,~266100,~People's Republic of China}

\email{gaoxiangshuli@126.com}

\author{Yu Zheng}


\address{Department of Mathematics,~East China Normal University,
~Lane 500, DongChuan Road,~Minhang District,~Shanghai
City,~200241,~People's Republic of China}

\email{zhyu@math.ecnu.edu.cn}

\date{May 16, 2011}

\begin{abstract}
In this paper,~we firstly establish an Interpolating curvature
invariance between the well known nonnegative and 2-non-negative
curvature invariant along the Ricci flow.~Then a related strong
maximum principle for the~$(\lambda_1, \lambda_2)$-nonnegativity is
also derived along Ricci flow.~Based on these,~finally we obtain a
rigidity property of manifolds with~$(\lambda_1,
\lambda_2)$-nonnegative curvature operators.
\end{abstract}

\maketitle

\section{Introduction And The Main Results}

One of the basic problems in Riemannian geometry is to relate
curvature and topology~(see [1], [2], [3] and [4]).~In [5] B\"ohm
and Wilking have proved that \emph{n}-dimensional closed Riemannian
manifolds with 2-positive curvature operators are diffeomorphic to
spherical space forms,~i.e.,~they admit metrics with constant
positive sectional curvature.~One of the key points of their theorem
is that the 2-positive or 2-nonnegative curvature condition is
preserved by the Ricci flow.

Recall that the Riemannian curvature tensor is defined by
$$
 Rm\left(
{X,Y} \right)Z = \nabla _X \nabla _Y Z - \nabla _Y \nabla _X Z -
\nabla _{\left[ {X,Y} \right]} Z.
$$
The Riemannian curvature operator,~denoted by~$ \mathcal {R} $,~is
the symmetric bilinear form on~$ \Lambda ^2 TM$~(or self-adjoint
transformation of~$ \Lambda ^2 TM$)~defined by
$$
\mathcal {R}\left( {X\Lambda Y,Z\Lambda W} \right) = \left\langle
{\mathcal {R}\left( {X\Lambda Y} \right),Z\Lambda W} \right\rangle =
2\left\langle {Rm\left( {X,Y} \right)W,Z} \right\rangle
$$
for tangent vectors~$X, Y, Z, W$.~Let~$ \left\{ {\mu _\alpha \left(
\mathcal {R} \right)\left| {\mu _1  \le  \cdots  \le \mu _N }
\right.} \right\}_{\alpha  = 1}^N$,~where $ N = {{n\left( {n - 1}
\right)} \mathord{\left/
 {\vphantom {{n\left( {n - 1} \right)} 2}} \right.
 \kern-\nulldelimiterspace} 2}$,~denote the eigenvalues of
Riemannian curvature operator~$ \mathcal {R} $.~We have the
following definition:
\begin{definition}[2-positive curvature operator]
A Riemannian manifold $ \left( {M^n ,g} \right)$~has 2-positive
curvature operator if
$$
 \mu _\alpha  \left(\mathcal {R} \right) + \mu
_\beta \left( \mathcal {R} \right) > 0 \eqno(1)
$$
for arbitrary~$ \alpha  \ne \beta$.
\end{definition}
{\it Remark 1.}~The 2-nonnegative curvature operator is defined in
the obvious way with~$ \ge $~replacing~$>$~in~(1).

In this paper,~we will consider a generalization of the
2-nonnegative curvature,~which is named as~$(\lambda_1,
\lambda_2)$-nonnegative curvature operator.~It relies on three
eigenvalues of Riemannian curvature operator~$ \mathcal {R} $. Let~$
\left\{ {\omega _\alpha } \right\}_{\alpha  = 1}^N$~be an
orthonormal basis of eigenvectors of~$ \mathcal {R} $~in~$so(n)$
with corresponding eigenvalues~$\mu _1 \le \mu _2  \le \cdots  \le
\mu _N$,~where~$ N = {{n\left( {n - 1} \right)} \mathord{\left/
 {\vphantom {{n\left( {n - 1} \right)} 2}} \right.
 \kern-\nulldelimiterspace} 2}$,~and let
$$
 \Lambda  = \left\{ {\left(
{x,y} \right) \in \left[ {0,1} \right] \times \left[ {0,1}
\right]\left| {0 \le y \le x \le 1,0 < 1 - \left( {x + y} \right)y
\le x \le 1} \right.} \right\},
$$
then as Definition 1.1,~the definition of $(\lambda_1,
\lambda_2)$-nonnegative curvature operator is given as follows:
\begin{definition}[$(\lambda_1, \lambda_2)$-nonnegative curvature operator]
For any 2 parameters~$(\lambda_1, \lambda_2) \in \Lambda $,~the
curvature operator~$\mathcal {R}$~is called as one~$(\lambda_1,
\lambda_2)$-nonnegative curvature operator on Riemannian
manifold~$(M^n, g)$~if~$ \mathcal {R} \in C_{\lambda_1,
\lambda_2}$,~where
$$
C_{\lambda _1 ,\lambda _2 }  = \left\{ {\mathcal {R}\left|
{\begin{array}{*{20}c}
   {C_{1,\lambda _1 ,\lambda _2 ,\alpha ,\beta ,\gamma }
   \left( \mathcal {R} \right) \ge 0}  \\
   {C_{2,\lambda _1 ,\lambda _2 ,\alpha ,\beta }
   \left( \mathcal {R} \right) \ge 0}  \\
\end{array},\begin{array}{*{20}c}
   {\left( {\lambda _1 ,\lambda _2 } \right) \in \Lambda }  \\
   {\forall 1 \le \alpha  < \beta  < \gamma  \le N}  \\
\end{array}} \right.} \right\}, \eqno(2)
$$
where
$$
 C_{1,\lambda _1 ,\lambda _2 ,\alpha ,\beta ,\gamma } \left( \mathcal {R} \right) =
\mathcal {R}_{\alpha \alpha }  + \lambda _1 \mathcal {R}_{\beta
\beta } + \lambda _2 \mathcal {R}_{\gamma \gamma },
$$
$$
 C_{2,\lambda
_1 ,\lambda _2 ,\alpha ,\beta } \left( \mathcal {R} \right) =
\lambda _1 \mathcal {R}_{\alpha \alpha } + \left( {1 - \left(
{\lambda _1  + \lambda _2 } \right)\lambda _2 } \right)\mathcal
{R}_{\beta \beta }
$$
and
$$
 \mathcal {R}_{\alpha \alpha } = \mathcal {R}\left( {\omega
_\alpha ,\omega _\alpha  } \right) = \mu _{\alpha }.
$$
\end{definition}

{\it Remark 2.}~The~$(\lambda_1, \lambda_2)$-positive curvature
operator~$ C_{\lambda _1 ,\lambda _2 }^ +$~is defined in the obvious
way with~$>$~replacing~$ \ge $~in~(2).

{\it Remark 3.}~It is easy to see that for the fixed~$ \lambda _1
$,~let
$$
\lambda _2 \to \frac{{\sqrt {\lambda _1^2  + 4}  - \lambda _1 }}{2}
$$
such that~$ \left( \lambda _1 ,\lambda _2 \right) \in \Lambda
$,~then we obtain that the~$ \left( \lambda _1 ,\lambda _2 \right)
$-nonnegative curvature operator turns into nonnegative curvature
operator~(see Theorem 2.2 for details).

Moreover it turns into 2-nonnegative curvature operator in [5] if~$
\lambda _1 = 1, \lambda _2 = 0 $.~In fact this curvature is an
Interpolating curvature condition between nonnegative and
2-nonnegative curvature.~Actually it can be seen in section 2 that
$$
\bigcap\limits_{\left( {\lambda _1 ,\lambda _2 } \right) \in \Lambda
} {C_{\lambda _1 ,\lambda _2 } }  = \left\{ {\mathcal {R}\left|
{\mathcal {R}_{\alpha \alpha }  \ge 0,\forall 1 \le \alpha  \le N}
\right.} \right\}
$$
and
$$
\bigcup\limits_{\left( {\lambda _1 ,\lambda _2 } \right) \in \Lambda
} {C_{\lambda _1 ,\lambda _2 } }  = \left\{ {\mathcal {R}\left|
{\mathcal {R}_{\alpha \alpha }  + \mathcal {R}_{\beta \beta }  \ge
0,\forall 1 \le \alpha < \beta \le N} \right.} \right\}.
$$

{\it Remark 4.}~But the~$ \left( {\lambda _1 ,\lambda _2 } \right)
$-nonnegative curvature operator is not always equal to
2-nonnegative curvature.~For example,
$$
 \left( {\lambda _1 ,\lambda
_2 } \right) = \left( {\frac{3}{4},\frac{1}{2}} \right),
$$
or more generally when~$ \lambda _1  > 1 - \left( {\lambda _1  +
\lambda _2 } \right)\lambda _2 $,~the curvature operator~$ \mathcal
{R}_{11}  = - 1,\mathcal {R}_{22} = 1, \cdots$~which satisfies
2-nonnegativity is not $ \left( {\lambda _1 ,\lambda _2 } \right)
$-nonnegativity.

Now we formulate one of the main results of this paper as follows:
\begin{theorem}[Weak maximum principle for~$ \left( {\lambda _1 ,\lambda _2 } \right)
$-nonnegativity] Let $ \left( {M^n ,g\left( t \right)} \right)$,~$ t
\in \left[ {0,T} \right)$,~be a solution to the Ricci flow on a
closed manifold.~If the curvature operator~$ \mathcal {R}(g(0))
$~is~$ \left( {\lambda _1 ,\lambda _2 } \right) $-nonnegative,~then
for any~$ 0 \le t < T $~the curvature operator~$ \mathcal {R}(g(t))
$~is also~$ \left( {\lambda _1 ,\lambda _2 } \right) $-nonnegative.
\end{theorem}
{\it Remark 5.}~Since when~$ \lambda _1 = 1 , \lambda _2 = 0
$,~the~$ \left( 1 , 0 \right) $-nonnegative curvature operator turns
into the well known 2-nonnegative curvature operator defined by
B\"ohm and Wilking in [5],~as a corollary of Theorem 1.3 we obtain
the invariance of 2-nonnegative curvature along the Ricci flow
again:~if the curvature operator~$ \mathcal {R}(g(0)) $~is
2-nonnegative,~then for any~$ t \ge 0$~the curvature operator~$
\mathcal {R}(g(t)) $~is also 2-nonnegative.

Moreover by using weak maximum principle for~$ \left( {\lambda _1
,\lambda _2 } \right) $-nonnegativity, we also derive a useful
strong maximum principle for~$ \left( {\lambda _1 ,\lambda _2 }
\right) $-nonnegativity:
\begin{theorem}[Strong maximum principle for~$ \left( {\lambda _1 ,\lambda _2
} \right) $-nonnegativity] Let $ \left( {M^n ,g\left( t \right)}
\right)$,~$ t \in \left[ {0,T} \right)$,~be a solution to the Ricci
flow on a closed manifold.~If the curvature operator~$ \mathcal
{R}(g(0)) $~is~$ \left( {\lambda _1 ,\lambda _2 } \right)
$-nonnegative,~then

\noindent \emph{(i)}~For any~$ 0 < t < T$~the curvature operator~$
\mathcal {R}(g(t)) $~is either nonnegative or~$ \left( {\lambda _1
,\lambda _2 } \right) $-positive.

\noindent \emph{(ii)}~If~~$
 C_{1,\lambda _1 ,\lambda _2 ,\alpha ,\beta
,\gamma } \left( \mathcal {R} (g(0)) \right) $~is positive at a
point~$ x_0 $~in~$ M^n $,~then
$$
 C_{1,\lambda _1
,\lambda _2 ,\alpha ,\beta ,\gamma } \left( \mathcal {R} (g(t))
\right) > 0
$$
everywhere for any~$ 0 < t < T$.~Furthermore,~if~~$
 C_{2,\lambda _1
,\lambda _2 ,\alpha ,\beta } \left( \mathcal {R} (g(0)) \right) $~is
positive at a point~$ x_0 $~in~$ M^n $,~then
$$
 C_{2,\lambda _1
,\lambda _2 ,\alpha ,\beta } \left( \mathcal {R} (g(t)) \right)
> 0
$$
is also positive everywhere for any~$ 0 < t < T$.
\end{theorem}
Then as a corollary of Theorem 1.4,~we also have the following
strong maximum principle for~$ \left( {\lambda _1 ,\lambda _2 }
\right) $-nonnegativity:
\begin{corollary}
Let~$ \left( {M^n ,g\left( t \right)} \right)$,~$ t \in \left[ {0,T}
\right)$,~be a solution to the Ricci flow on a closed
manifold.~Suppose that the curvature operator~$ \mathcal {R}(g(0)) $
is~$ \left( {\lambda _1 ,\lambda _2 } \right) $-nonnegative,~if~$
\mathcal {R}(g(0)) $~is~$ \left( {\lambda _1 ,\lambda _2 } \right)
$-positive at a point~$ x_0 $~in~$ M^n $, then~$ \mathcal {R}(g(t))
$ is $ \left( {\lambda _1 ,\lambda _2 } \right) $-positive
everywhere for any~$ 0 < t < T$.
\end{corollary}

On the other hand,~in [5] B\"ohm and Wilking also derive a
convergence result of 2-positive curvature along the Ricci flow:~on
a compact manifold the normalized Ricci flow evolves a Riemannian
metric with 2-positive curvature operator to a limit metric with
constant sectional curvature.

Then by using
$$
\bigcup\limits_{\left( {\lambda _1 ,\lambda _2 } \right) \in \Lambda
} {C_{\lambda _1 ,\lambda _2 }^ +  }  = \left\{ {\mathcal {R}\left|
{\mathcal {R}_{\alpha \alpha }  + \mathcal {R}_{\beta \beta }  >
0,\forall 1 \le \alpha  < \beta \le N} \right.} \right\}
$$
which is proved in section 2 we have
$$
C_{\lambda _1 ,\lambda _2 }^ + \subset \left\{ {\mathcal {R}\left|
{\mathcal {R}_{\alpha \alpha }  + \mathcal {R}_{\beta \beta }  >
0,\forall 1 \le \alpha < \beta  \le N} \right.} \right\}
$$
for arbitrary~$ \left( {\lambda _1 ,\lambda _2 } \right) \in \Lambda
$.~Thus by using Theorem 1.4 we derive a convergence result of~$
\left( {\lambda _1 ,\lambda _2 } \right) $-positive curvature along
the Ricci flow:
\begin{corollary}
Let~$ \left( {M^n ,g\left( t \right)} \right)$,~$ t \in \left[
{0,\infty } \right)$,~be a solution to the normalized Ricci flow:
$$
\frac{{\partial g}}{{\partial t}} =  - 2Rc\left( g \right) +
\frac{2}{n}rg
$$
on the closed manifold~$ M^n $~and~$ \left( {\lambda _1 ,\lambda _2
} \right) \in \Lambda $.~If the curvature operator~$ \mathcal
{R}(g(0)) $ is~$ \left( {\lambda _1 ,\lambda _2 } \right)
$-nonnegative and~$ \left( {\lambda _1 ,\lambda _2 } \right)
$-positive at a point in~$ M^n $,~then the normalized Ricci flow
evolves the Riemannian metric to a limit metric with constant
sectional curvature.
\end{corollary}

The paper is organized as follows.~In section 2,~we present some
preliminaries.~In section 3,~we prove Theorem 1.3 by directly
calculating.~In section 4,~we prove the strong maximum principle for
the~$ \left( {\lambda _1 ,\lambda _2 } \right) $-nonnegativity along
Ricci flow.~Based on this,~we also obtain a rigidity property of
manifolds with~$ \left( {\lambda _1 ,\lambda _2 } \right)
$-nonnegative curvature operators.

\section{Preliminaries}

Recall that for an orthonormal basis~$ \left\{ {\varphi ^\alpha  }
\right\}_{\alpha = 1}^N$~of~$ \Lambda ^2 T^* M^n \cong so\left( n
\right)$, where $so\left( n \right)$~denotes the anti-symmetric
matrix algebra and the structure constants for the Lie bracket are
given by
$$
 \left[ {\varphi ^\alpha ,\varphi ^\beta  } \right] =
\sum\limits_\gamma {c_\gamma ^{\alpha \beta } \varphi ^\gamma }.
$$
So that
$$
 c_\gamma ^{\alpha \beta }  = \left\langle {\left[
{\varphi ^\alpha  ,\varphi ^\beta } \right],\varphi ^\gamma },
\right\rangle
$$
which is anti-symmetric in all 3 components.

For the Lie algebra~$ \Lambda ^2 T^{*} M^n$,~let~$ \mathcal
{F}_{\Lambda ^2 TM^n } $~denote the symmetric bilinear functional
on~$ \Lambda ^2 TM^n$,~then we have the following sharp product
operator~$ \# :\mathcal {F}_{\Lambda ^2 TM^n }  \times \mathcal
{F}_{\Lambda ^2 TM^n } \to \mathcal {F}_{\Lambda ^2 TM^n } $:
$$
\left( {A\# B} \right)_{\alpha \beta }  = \left( {A\# B}
\right)\left( {\varphi _\alpha  ,\varphi _\beta  } \right) =
\frac{1}{2}c_\alpha ^{\gamma \eta } c_\beta ^{\delta \theta }
A_{\gamma \delta } B_{\eta \theta },
$$
where~$ \left\{ {\varphi _\alpha } \right\}_{\alpha  = 1}^N$~is the
the dual orthonormal basis~of~$ \Lambda ^2 TM^n$,~ \emph{A} and
\emph{B} are the symmetric bilinear functionals on~$ \Lambda ^2
TM^n$~such that~$ A_{\gamma \delta } = A\left( {\varphi _\gamma
,\varphi _\delta  } \right)$~and $ B_{\eta \theta } = B\left(
{\varphi _\eta ,\varphi _\theta  } \right)$,~and also let~$ A^\#
$~denote~$ A\# A $.

In this paper,~we need to use the following famous maximum principle
for symmetric 2-tensors proved by Hamilton~[7]:
\begin{lemma}[Maximum principle for symmetric 2-tensors]
Let~$ g\left( t \right) $~be a smooth 1-parameter family of
Riemannian metrics on a closed manifold~$ {M^n} $. Let~$ \alpha
\left( t \right)$~be a symmetric 2-tensor satisfying
$$
\frac{\partial }{{\partial t}}\alpha  \ge \Delta _{g\left( t
\right)} \alpha  + \nabla _{X\left( t \right)} \alpha  + \beta,
$$
where~$ X\left( t \right) $~is a time-dependent vector field and
$$
\beta \left( {x,t} \right) = \beta \left( {\alpha \left( {x,t}
\right),g\left( {x,t} \right)} \right)
$$
is a symmetric 2-tensor which is locally Lipschitz in all its
arguments.~Suppose~$ \beta $~satisfies the null-eigenvector
assumption and that if~$ A_{ij} $~is a nonnegative symmetric
2-tensor at a point~$ \left( {x,t} \right) $~and if V is a vector
such that~$ A_{ij} V^j  = 0$,~then
$$
\beta _{ij} \left( {A,g} \right)V^i V^j  \ge 0.
$$
If~$ \alpha \left( 0 \right) \ge 0$,~then~$ \alpha \left( t \right)
\ge 0$~for all~$ t \ge 0$~as long as the solution exists.
\end{lemma}

Moreover for the~$ \left( {\lambda _1 ,\lambda _2 } \right)
$-nonnegative curvature operator~$ C_{\lambda _1 ,\lambda _2 } $,~we
have the following interesting property:
\begin{theorem}
Let
$$
\Lambda  = \left\{ {\left( {x,y} \right) \in \left[ {0,1} \right]
\times \left[ {0,1} \right]\left| {0 \le y \le x \le 1,0 < 1 -
\left( {x + y} \right)y \le x \le 1} \right.} \right\},
$$
then we have
$$
\bigcap\limits_{\left( {\lambda _1 ,\lambda _2 } \right) \in \Lambda
} {C_{\lambda _1 ,\lambda _2 } }  = \left\{ {\mathcal {R}\left|
{\mathcal {R}_{\alpha \alpha }  \ge 0,\forall 1 \le \alpha  \le N}
\right.} \right\}, \eqno(3)
$$
$$
\bigcup\limits_{\left( {\lambda _1 ,\lambda _2 } \right) \in \Lambda
} {C_{\lambda _1 ,\lambda _2 } }  = \left\{ {\mathcal {R}\left|
{\mathcal {R}_{\alpha \alpha }  + \mathcal {R}_{\beta \beta }  \ge
0,\forall 1 \le \alpha < \beta \le N} \right.} \right\}. \eqno(4)
$$
and
$$
\bigcup\limits_{\left( {\lambda _1 ,\lambda _2 } \right) \in \Lambda
} {C_{\lambda _1 ,\lambda _2 }^ +  }  = \left\{ {\mathcal {R}\left|
{\mathcal {R}_{\alpha \alpha }  + \mathcal {R}_{\beta \beta }  >
0,\forall 1 \le \alpha  < \beta \le N} \right.} \right\}\eqno(5)
$$
\end{theorem}
\begin{proof}
Firstly we prove~(3).~For any
$$
 \mathcal {R} \in \left\{ {\mathcal
{R}\left| {\mathcal {R}_{\alpha \alpha }  \ge 0,\forall 1 \le \alpha
\le N} \right.} \right\},
$$
we have
$$
{\mathcal {R}_{\alpha \alpha }  + \lambda _1 \mathcal {R}_{\beta
\beta }  + \lambda _2 \mathcal {R}_{\gamma \gamma }  \ge 0}
$$
and
$$
{\lambda _1 \mathcal {R}_{\alpha \alpha }  + \left( {1 - \left(
{\lambda _1  + \lambda _2 } \right)\lambda _2 } \right)\mathcal
{R}_{\beta \beta }  \ge 0}
$$
satisfy for arbitrary~$ {1 \le \alpha  < \beta  < \gamma  \le
N}$~and~$ \left( {\lambda _1 ,\lambda _2 } \right) \in \Lambda
$,~which implies~$ \mathcal {R} \in C_{\lambda _1 ,\lambda _2
}$.~Thus
$$
\left\{ {\mathcal {R}\left| {\mathcal {R}_{\alpha \alpha }  \ge
0,\forall 1 \le \alpha \le N} \right.} \right\} \subset C_{\lambda
_1 ,\lambda _2 }\eqno(6)
$$
for arbitrary~$ \left( {\lambda _1 ,\lambda _2 } \right) \in \Lambda
$.~Conversely,~we consider a fixed~$ {\overline {\lambda }_1 }
$,~since
$$
 0 \le \frac{{\sqrt {\overline \lambda  _1^2  + 4} -
\overline \lambda  _1 }}{2} = \frac{2}{{\sqrt {\overline \lambda
_1^2  + 4} + \overline \lambda  _1 }} \le 1
$$
and for sufficiently large~$ \overline \lambda  _1 $
$$
\frac{2}{{\sqrt {\overline \lambda  _1^2  + 4}  + \overline \lambda
_1 }} < \overline \lambda  _1,
$$
we have~$ \left( {\overline {\lambda _1 } ,\lambda _2 } \right) \in
\Lambda$~if
$$
\left| {\lambda _2  - \frac{{\sqrt {\overline \lambda  _1^2  + 4}  -
\overline \lambda  _1 }}{2}} \right|
$$
is sufficiently small.~Note that for any~$ \mathcal {R} \in
\bigcap\limits_{\left( {\lambda _1 ,\lambda _2 } \right) \in \Lambda
} {C_{\lambda _1 ,\lambda _2 } }$,~we have~$ \mathcal {R} \in
C_{\overline {\lambda }_1 ,\lambda _2 } $~if~$ \left( {\overline
{\lambda }_1 ,\lambda _2 } \right) \in \Lambda$,~it follows that
$$
\overline \lambda  _1 \mathcal {R}_{\alpha \alpha }  + \left( {1 -
\left( {\overline \lambda  _1  + \lambda _2 } \right)\lambda _2 }
\right)\mathcal {R}_{\beta \beta }  \ge 0
$$
for arbitrary~$ {1 \le \alpha  < \beta  \le N}$,~which implies that
$$
\mathcal {R}_{11}  \ge  - \frac{{1 - \left( {\overline \lambda  _1 +
\lambda _2 } \right)\lambda _2 }}{{\overline \lambda  _1 }}\mathcal
{R}_{22}
$$
for the above arbitary~$ \left( {\overline {\lambda _1 } ,\lambda _2
} \right) \in \Lambda$.~Then let
$$
 \lambda _2  \to \frac{{\sqrt {\overline \lambda  _1^2  + 4}  - \overline \lambda  _1
}}{2}
$$
such that~$ \left( {\overline {\lambda }_1 ,\lambda _2 } \right) \in
\Lambda$,~we have
$$
1 - \left( {\overline \lambda  _1  + \lambda _2 } \right)\lambda _2
\to 0,
$$
which implies~$ \mathcal {R}_{11}  \ge 0$.~Since~$ \mathcal {R}_{11}
\le \mathcal {R}_{22} \le \cdots \le \mathcal {R}_{NN}$,~it follows
that
$$
\mathcal {R} \in \left\{ {\mathcal {R}\left| {\mathcal {R}_{\alpha
\alpha } \ge 0,\forall 1 \le \alpha  \le N} \right.} \right\}.
$$
Thus
$$
\bigcap\limits_{\left( {\lambda _1 ,\lambda _2 } \right) \in \Lambda
} {C_{\lambda _1 ,\lambda _2 } }  \subset  \left\{ {\mathcal
{R}\left| {\mathcal {R}_{\alpha \alpha }  \ge 0,\forall 1 \le \alpha
\le N} \right.} \right\}. \eqno(7)
$$
By using~(6)~and~(7)~we complete the proof of~(3).

Then we prove~(4).~For any~$ \left( {\lambda _1 ,\lambda _2 }
\right) \in \Lambda $,~if~$ \mathcal {R} \in C_{\lambda _1 ,\lambda
_2 }$~we have
$$
{\lambda _1 \mathcal {R}_{\alpha \alpha }  + \left( {1 - \left(
{\lambda _1  + \lambda _2 } \right)\lambda _2 } \right)\mathcal
{R}_{\beta \beta }  \ge 0}
$$
for arbitrary~$ {1 \le \alpha  < \beta  \le N}$.~By the definition
of~$\Lambda $~we have
\[
\begin{split}
 \mathcal {R}_{\alpha \alpha }  + \mathcal {R}_{\beta \beta } & = \frac{1}{{\lambda _1 }}\left( {\lambda _1 \mathcal {R}_{\alpha \alpha }  + \lambda _1 \mathcal {R}_{\beta \beta } } \right) \\
 & \ge \lambda _1 \mathcal {R}_{\alpha \alpha }  + \left( {1 - \left( {\lambda _1  + \lambda _2 } \right)\lambda _2 } \right)\mathcal {R}_{\beta \beta }  \\
 & \ge 0. \\
 \end{split}
\]
Hence
$$
C_{\lambda _1 ,\lambda _2 }  \subset \left\{ {\mathcal {R}\left|
{\mathcal {R}_{\alpha \alpha }  + \mathcal {R}_{\beta \beta }  \ge
0,\forall 1 \le \alpha < \beta \le N} \right.} \right\}
$$
for any~$ \left( {\lambda _1 ,\lambda _2 } \right) \in \Lambda
$,~which implies
$$
\bigcup\limits_{\left( {\lambda _1 ,\lambda _2 } \right) \in \Lambda
} {C_{\lambda _1 ,\lambda _2 } }  \subset \left\{ {\mathcal
{R}\left| {\mathcal {R}_{\alpha \alpha }  + \mathcal {R}_{\beta
\beta }  \ge 0,\forall 1 \le \alpha < \beta  \le N} \right.}
\right\}. \eqno(8)
$$
Conversely,~when~$ \lambda _1 = 1 , \lambda _2 = 0 $,~the~$ \left( 1
, 0 \right) $-nonnegative curvature operator
$$
C_{1,0}  = \left\{ {\mathcal {R}\left| {\mathcal {R}_{\alpha \alpha
}  + \mathcal {R}_{\beta \beta } \ge 0,\forall 1 \le \alpha  < \beta
\le N} \right.} \right\} ,
$$
which implies
$$
\bigcup\limits_{\left( {\lambda _1 ,\lambda _2 } \right) \in \Lambda
} {C_{\lambda _1 ,\lambda _2 } }  \supset C_{1,0}  = \left\{
{\mathcal {R}\left| {\mathcal {R}_{\alpha \alpha }  + \mathcal
{R}_{\beta \beta }  \ge 0,\forall 1 \le \alpha  < \beta  \le N}
\right.} \right\}. \eqno(9)
$$
By using~(8)~and~(9)~we complete the proof of~(4).

The proof of~(5)~is similar with the proof of~(4).~We only need to
note
\[
\begin{split}
 \mathcal {R}_{\alpha \alpha }  + \mathcal {R}_{\beta \beta } & = \frac{1}{{\lambda _1 }}\left( {\lambda _1 \mathcal {R}_{\alpha \alpha }  + \lambda _1 \mathcal {R}_{\beta \beta } } \right) \\
 & \ge \lambda _1 \mathcal {R}_{\alpha \alpha }  + \left( {1 - \left( {\lambda _1  + \lambda _2 } \right)\lambda _2 } \right)\mathcal {R}_{\beta \beta }  \\
 & > 0, \\
 \end{split}
\]
and when~$ \lambda _1 = 1 , \lambda _2 = 0 $,~the~$ \left( 1 , 0
\right) $-positive curvature operator
$$
C_{1,0}  = \left\{ {\mathcal {R}\left| {\mathcal {R}_{\alpha \alpha
}  + \mathcal {R}_{\beta \beta }
> 0,\forall 1 \le \alpha  < \beta  \le N} \right.} \right\} .
$$
\end{proof}

\section{Lemmas And Proof Of Theorem 1.3}

In this section,~we prove Theorem 1.3.~Firstly recall that the
following maximum principle established by Hamilton,~Chow and
Lu~(see [6] and [7])~is very useful in the Ricci flow:
\begin{lemma}[Maximum principle for convex sets]
Let~$ \left( {M^n ,g\left( t \right)} \right)$~be a solution to the
Ricci flow and let~$ K\left( t \right) \subset E = \Lambda ^2 M^n
\otimes _S \Lambda ^2 M^n$~be subsets which are invariant under
parallel translation and whose intersections~$ K\left( t \right)_x =
K\left( t \right) \cap E_x$~with each fiber are closed and
convex.~Suppose also that the set~$ \left\{ {\left( {v,t} \right)
\in E \times \left[ {0,T} \right)\left| {v \in K\left( t \right)}
\right.} \right\}$~is closed in~$ E \times \left[ {0,T} \right)$~and
suppose the \emph{ODE}~$ {{dM} \mathord{\left/
 {\vphantom {{dM} {dt}}} \right.
 \kern-\nulldelimiterspace} {dt}} = M^2  + M^\#$~has the property that
for any~$ M\left( {t_0 } \right) \in K\left( {t_0 } \right)$,~we
have~$ M\left( t \right) \in K\left( t \right)$~for arbitrary~$ t
\in \left[ {t_0 ,T} \right)$.~Then if~$ \mathcal {R}\left( 0 \right)
\in K\left( 0 \right)$,~we have~$ \mathcal {R}\left( t \right) \in
K\left( t \right)$~for arbitrary~$ t \in \left[ {0 ,T} \right)$.
\end{lemma}
Then we present the following lemma which is used in the proof of
Theorem 1.3:
\begin{lemma}
Suppose~$ \left( {M^n ,g} \right)$~is a Riemannian manifold,~and
let~$ C_{\lambda _1 ,\lambda _2 }$ denote the convex cone of~$
\left( {\lambda _1 ,\lambda _2 } \right) $-nonnegative curvature
operators.~Then $ C_{\lambda _1 ,\lambda _2 }$~is preserved by the
\emph{ODE}~$ {{d\mathcal {R}} \mathord{\left/
 {\vphantom {{d\mathcal {R}} {dt}}} \right.
 \kern-\nulldelimiterspace} {dt}} = \mathcal {R}^2  + \mathcal {R}^\#$.
\end{lemma}
\begin{proof}
Note that the convexity and invariance under parallel translation
of~$ C_{\lambda _1 ,\lambda _2 }$~are obvious~(see [6]).~Hence we
only need to show that~$ C_{\lambda _1 ,\lambda _2 }$~is preserved
by the ODE~$ {{d\mathcal {R}} \mathord{\left/
 {\vphantom {{d\mathcal {R}} {dt}}} \right.
 \kern-\nulldelimiterspace} {dt}} = \mathcal {R}^2  + \mathcal {R}^\#$.~

Let~$ \left\{ {\omega _\alpha  } \right\}_{\alpha  = 1}^N$~be an
orthonormal basis of eigenvectors of~$ \mathcal {R}
$~in~$so(n)$~with corresponding eigenvalues~$\mu _1 (\mathcal {R})
\le \mu _2 (\mathcal {R}) \le \cdots  \le \mu _N(\mathcal {R})
$,~where~$ N = {{n\left( {n - 1} \right)} \mathord{\left/
 {\vphantom {{n\left( {n - 1} \right)} 2}} \right.
 \kern-\nulldelimiterspace} 2}$.~If there exists some point~$ x_0 \in M^n $~and time~$ t_0 $~such
that $\mu_1 (\mathcal {R} (x_0, t_0)) > 0$,~and note that it is
obvious that~$ \mathcal {R}^2  + \mathcal {R}^\# $~satisfies the
null-eigenvector assumption~(see [6]),~then by using the maximum
principle Lemma 2.1 we have
$$
0 \le \mu _1 \left( \mathcal {R} (x, t) \right) \le \mu _2 \left(
\mathcal {R} (x, t) \right) \le \cdots \le \mu _N \left( \mathcal
{R} (x, t) \right)
$$
for any~$ (x, t) \in M^n \times \left[ {t_0, T } \right) $,~which is
equivalent to~$ \mathcal {R}\left( t \right) \ge 0$~when~$ t \ge t_0
$.~Then by Theorem 2.2,~we have~$ \mathcal {R}\left( t \right)$~is~$
\left( {\lambda _1 ,\lambda _2 } \right) $-nonnegative,~which
completes the proof.

Hence in the proof we only need to consider the time interval such
that $\mu_1 (\mathcal {R} (x, t)) \le 0$.~If for this case we can
also prove that~$ \mathcal {R}\left( t \right)$~is $ \left( {\lambda
_1 ,\lambda _2 } \right) $-nonnegative,~then by the discussion
before we complete the proof.

Let~$ \mu _\alpha = \mathcal {R}\left( {\omega _\alpha ,\omega
_\alpha } \right) = \mathcal {R}_{\alpha \alpha } $,~then it follows
from~$\mu _1 (\mathcal {R}) \le \mu _2 (\mathcal {R}) \le \cdots \le
\mu _N(\mathcal {R}) $~that~$ \mathcal {R}_{11}  \le \mathcal
{R}_{22} \le \cdots \le \mathcal {R}_{NN}$.~Thus
$$
\begin{array}{l}
 \quad \mathcal {R}_{\alpha \alpha } \left( t \right) + \lambda _1 \mathcal {R}_{\beta \beta } \left( t \right) + \lambda _2 \mathcal {R}_{\gamma \gamma } \left( t \right) \\
\\
  \ge \mathcal {R}_{11} \left( t \right) + \lambda _1 \mathcal {R}_{22} \left( t \right) + \lambda _2 \mathcal {R}_{33} \left( t \right) \\
 \end{array}
$$
and
$$
\begin{array}{l}
 \quad \lambda _1 \mathcal {R}_{\alpha \alpha } \left( t \right) + \left( {1 - \left( {\lambda _1  + \lambda _2 } \right)\lambda _2 } \right)\mathcal {R}_{\beta \beta } \left( t \right) \\
\\
  \ge \lambda _1 \mathcal {R}_{11} \left( t \right) + \left( {1 - \left( {\lambda _1  + \lambda _2 } \right)\lambda _2 } \right)\mathcal {R}_{22} \left( t \right) \\
 \end{array}
$$
for arbitrary~$ t \ge 0 $~and~$ {1 \le \alpha  < \beta  < \gamma \le
N}$.

Hence we only need to prove
$$
   \mathcal {R}_{11} \left( t \right) + \lambda _1 \mathcal {R}_{22} \left( t \right) + \lambda _2 \mathcal {R}_{33} \left( t \right) \ge 0
$$
and
$$
   \lambda _1 \mathcal {R}_{11} \left( t \right) + \left( {1 - \left( {\lambda _1  + \lambda _2 } \right)\lambda _2 } \right)\mathcal {R}_{22} \left( t \right)
  \ge 0
$$
are preserved by the ODE~$ {{d\mathcal {R}} \mathord{\left/
 {\vphantom {{d\mathcal {R}} {dt}}} \right.
 \kern-\nulldelimiterspace} {dt}} = \mathcal {R}^2  + \mathcal {R}^\#$.~This
is equivalent to that~$ R^2  + R^\# $~lies inside the tangent cone
of the convex cone~$ C_{\lambda _1 ,\lambda _2 }$~for~$ R \in
\partial C_{\lambda _1 ,\lambda _2 }$.~Thus we only need to prove the following claim:
\begin{claim}
\emph{(i)}~If
$$
 \mathcal {R}_{\alpha \alpha }  + \lambda _1 \mathcal {R}_{\beta \beta }  +
\lambda _2 \mathcal {R}_{\gamma \gamma }  = 0
$$
and
$$
 \lambda _1 \mathcal {R}_{\alpha \alpha } + \left( {1 - \left( {\lambda
_1 + \lambda _2 } \right)\lambda _2 } \right)\mathcal {R}_{\beta
\beta }  \ge 0
$$
for arbitrary~$ {1 \le \alpha  < \beta  < \gamma  \le N}$~and~$
\left( {\lambda _1 ,\lambda _2 } \right) \in \Lambda $,~then
$$
\left( {\mathcal {R}^2  + \mathcal {R}^\#  } \right)_{1 1}  +
\lambda _1 \left( {\mathcal {R}^2  + \mathcal {R}^\# } \right)_{2 2}
+ \lambda _2 \left( {\mathcal {R}^2  + \mathcal {R}^\#  } \right)_{3
3} \ge 0,
$$
where~$ \left( {\mathcal {R}^2  + \mathcal {R}^\#  } \right)_{\alpha
\alpha } = \left( {\mathcal {R}^2  + \mathcal {R}^\#  } \right)
\left( {\omega _\alpha ,\omega _\alpha } \right) $.
\\

\noindent \emph{(ii)}~If
$$
 \mathcal {R}_{\alpha \alpha }  + \lambda _1 \mathcal {R}_{\beta \beta }  +
\lambda _2 \mathcal {R}_{\gamma \gamma }  \ge 0
$$
and
$$
 \lambda _1 \mathcal {R}_{\alpha \alpha } + \left( {1 - \left( {\lambda
_1 + \lambda _2 } \right)\lambda _2 } \right)\mathcal {R}_{\beta
\beta } = 0
$$
for arbitrary~$ {1 \le \alpha  < \beta  < \gamma  \le N}$~and~$
\left( {\lambda _1 ,\lambda _2 } \right) \in \Lambda $,~then
$$
\lambda _1 \left( {\mathcal {R}^2  + \mathcal {R}^\#  } \right)_{1
1}  + \left( {1 - \left( {\lambda _1  + \lambda _2 } \right)\lambda
_2 } \right)\left( {\mathcal {R}^2 + \mathcal {R}^\#  } \right)_{2
2}  \ge 0.
$$
\end{claim}
\begin{proof}[Proof of Claim 3.3]
(i)~By calculating
$$
\begin{array}{l}
 \quad \left( {\mathcal {R}^2  + \mathcal {R}^\#  } \right)_{11}  + \lambda _1 \left( {\mathcal {R}^2  + \mathcal {R}^\#  } \right)_{22}  + \lambda _2 \left( {\mathcal {R}^2  + \mathcal {R}^\#  } \right)_{33} \\
\\
  = \mu _1^2  + \lambda _1 \mu _2^2  + \lambda _2 \mu _3^2  + 2\sum\limits_{\alpha  < \beta } {\left( {\left( {c_1^{\alpha \beta } } \right)^2  + \lambda _1 \left( {c_2^{\alpha \beta } } \right)^2  + \lambda _2 \left( {c_3^{\alpha \beta } } \right)^2  } \right)} \mu _\alpha  \mu _\beta   \\
 \end{array} \eqno(10)
$$
We only need to prove the right part of~(10)~is nonnegative.~In fact
$$
\begin{array}{l}
 \quad \sum\limits_{\alpha  < \beta } {\left( {\left( {c_1^{\alpha \beta } } \right)^2  + \lambda _1 \left( {c_2^{\alpha \beta } } \right)^2  + \lambda _2 \left( {c_3^{\alpha \beta } } \right)^2 } \right)} \mu _\alpha  \mu _\beta   \\
  = \sum\limits_{2 \le \alpha  < \beta } {\left( {c_1^{\alpha \beta } } \right)^2 } \mu _\alpha  \mu _\beta   + \sum\limits_{1 \le \alpha  < \beta } {\lambda _1 \left( {c_2^{\alpha \beta } } \right)^2 } \mu _\alpha  \mu _\beta   + \sum\limits_{1 \le \alpha  < \beta } {\lambda _2 \left( {c_3^{\alpha \beta } } \right)^2 } \mu _\alpha  \mu _\beta   \\
  = \sum\limits_{\beta  > 2} {\left( {c_1^{2\beta } } \right)^2 } \mu _2 \mu _\beta   + \lambda _1 \sum\limits_{\beta  > 2} {\left( {c_2^{1\beta } } \right)^2 } \mu _1 \mu _\beta   + \lambda _2 \left( {c_3^{12} } \right)^2 \mu _1 \mu _2 \\
  \quad + \lambda _2 \sum\limits_{\beta  > 3} {\left( {c_3^{1\beta } } \right)^2 } \mu _1 \mu _\beta + \lambda _2 \sum\limits_{\beta  > 3} {\left( {c_3^{2\beta } } \right)^2 } \mu _2 \mu _\beta  + \sum\limits_{\beta  > 3} {\left( {c_1^{3\beta } } \right)^2 } \mu _3 \mu _\beta \\
  \quad + \sum\limits_{4 \le \alpha  < \beta } {\left( {c_1^{\alpha \beta } } \right)^2 } \mu _\alpha  \mu_\beta + \lambda _1 \sum\limits_{\beta  > 3} {\left( {c_2^{3\beta } }\right)^2 } \mu _3 \mu_\beta + \lambda _1 \sum\limits_{4 \le \alpha  < \beta } {\left( {c_2^{\alpha \beta } } \right)^2 } \mu _\alpha  \mu _\beta \\
  \quad + \lambda _2 \sum\limits_{4 \le \alpha  < \beta } {\left( {c_3^{\alpha \beta } } \right)^2 } \mu _\alpha  \mu _\beta   \\
  = \sum\limits_{\beta  > 3} {\left( {c_1^{2\beta } } \right)^2 } \left( {\mu _2  + \lambda _1 \mu _1 } \right)\mu _\beta   + \left( {c_1^{23} } \right)^2 \left( {\mu _2 \mu _3  + \lambda _1 \mu _1 \mu _3  + \lambda _2 \mu _1 \mu _2 } \right) \\
   \quad  + \sum\limits_{\beta  > 3} {\left( {c_1^{3\beta } } \right)^2 } \left( {\mu _3  + \lambda _2 \mu _1 } \right)\mu _\beta  + \sum\limits_{\beta  > 3} {\left( {c_2^{3\beta } } \right)^2 } \left( {\lambda _1 \mu _3  + \lambda _2 \mu _2 } \right)\mu _\beta \\
   \quad  + \sum\limits_{4 \le \alpha  < \beta } {\left( {c_1^{\alpha \beta } } \right)^2 } \mu _\alpha  \mu _\beta   + \lambda _1 \sum\limits_{4 \le \alpha  < \beta } {\left( {c_2^{\alpha \beta } } \right)^2 } \mu _\alpha  \mu _\beta   + \lambda _2 \sum\limits_{4 \le \alpha  < \beta } {\left( {c_3^{\alpha \beta } } \right)^2 } \mu _\alpha  \mu _\beta .  \\
 \end{array}
$$
It follows from the definition of~$ \Lambda $~that
$$
 \mu _2  +
\lambda _1 \mu _1 \ge \left( {1 - \left( {\lambda _1 + \lambda _2 }
\right)\lambda _2 } \right)\mu _2 + \lambda _1 \mu _1 \ge 0,
$$
thus
$$
\sum\limits_{\beta  > 3} {\left( {c_1^{2\beta } } \right)^2 } \left(
{\mu _2  + \lambda _1 \mu _1 } \right)\mu _\beta   \ge 0. \eqno(11)
$$
Since
$$
\begin{array}{l}
 \quad \mu _2 \mu _3  + \lambda _1 \mu _1 \mu _3  + \lambda _2 \mu _1 \mu
_2 \\
\\
= \lambda _2 \mu _2 \left( {\mu _1  + \left( {\lambda _1 + \lambda
_2 } \right)\mu _3 } \right) + \mu _3 \left( {\left( {1 - \left(
{\lambda _1  + \lambda _2 } \right)\lambda _2 } \right)\mu _2 +
\lambda _1 \mu _1 } \right)\\
 \ge \lambda _2 \mu _2 \left( {\mu _1  + \lambda _1 \mu _2  +
\lambda _2 \mu _3 } \right) + \mu _3 \left( {\left( {1 - \left(
{\lambda _1  + \lambda _2 } \right)\lambda _2 } \right)\mu _2 +
\lambda _1 \mu _1 } \right)\\
 \ge 0,
 \end{array}
$$
it follows that
$$
\left( {c_1^{23} } \right)^2 \left( {\mu _2 \mu _3 + \lambda _1 \mu
_1 \mu _3  + \lambda _2 \mu _1 \mu _2 } \right) \ge 0.  \eqno(12)
$$

Moreover,~by the discussion at the beginning of the proof,~we only
need to consider the time interval such that~$\mu_1 (\mathcal {R}
(x, t)) \le 0$.~Then it follows from the definition of~$ \left(
{\lambda _1 ,\lambda _2 } \right) $-nonnegative curvature operators
that we actually have
$$
 \mu _1(\mathcal {R}
(x, t)) \le 0 \le \mu _2(\mathcal {R} (x, t)) \le \cdots \le
\mu_N(\mathcal {R} (x, t)) ,
$$
then
$$
 \mu _3 + \lambda _2 \mu _1 \ge \mu _3 + \lambda _1 \mu _1  \ge
\mu _2  + \lambda _1 \mu _1 \ge 0,
$$
which is to say
$$
\sum\limits_{\beta  > 3} {\left( {c_1^{3\beta } } \right)^2 } \left(
{\mu _3  + \lambda _2 \mu _1 } \right)\mu _\beta \ge 0. \eqno(13)
$$
We also have
$$
\begin{array}{l}
 \quad \sum\limits_{\beta  > 3} {\left( {c_2^{3\beta } } \right)^2 } \left( {\lambda _1 \mu _3  + \lambda _2 \mu _2 } \right)\mu _\beta   + \sum\limits_{4 \le \alpha  < \beta } {\left( {c_1^{\alpha \beta } } \right)^2 } \mu _\alpha  \mu _\beta \\
 \quad + \lambda _1 \sum\limits_{4 \le \alpha  < \beta } {\left( {c_2^{\alpha \beta } } \right)^2 } \mu _\alpha  \mu _\beta + \lambda _2 \sum\limits_{4 \le \alpha  < \beta } {\left( {c_3^{\alpha \beta } } \right)^2 } \mu _\alpha  \mu _\beta  \ge 0. \\
 \end{array} \eqno(14)
$$
Hence~(11),~(12),~(13)~and~(14)~lead to
$$
\sum\limits_{\alpha  < \beta } {\left( {\left( {c_1^{\alpha \beta }
} \right)^2  + \lambda _1 \left( {c_2^{\alpha \beta } } \right)^2  +
\lambda _2 \left( {c_3^{\alpha \beta } } \right)^2  } \right)} \mu
_\alpha  \mu _\beta \ge 0
$$.

\noindent (ii)~Let~$ \gamma  = \lambda _1$~and~$ \delta  = 1 -
\left( {\lambda _1  + \lambda _2 } \right)\lambda _2$,~by
calculating
$$
\begin{array}{l}
 \quad \lambda _1 \left( {\mathcal {R}^2  + \mathcal {R}^\#  } \right)_{11}  + \left( {1 - \left( {\lambda _1  + \lambda _2 } \right)\lambda _2 } \right)\left( {\mathcal {R}^2  + \mathcal {R}^\#  } \right)_{22}  \\
  = \gamma \mu _1^2  + \delta \mu _2^2  + 2\sum\limits_{\alpha  < \beta } {\left( {\gamma \left( {c_1^{\alpha \beta } } \right)^2  + \delta \left( {c_2^{\alpha \beta } } \right)^2 } \right)} \mu _\alpha  \mu _\beta   \\
 \end{array} \eqno(15)
$$
We only need to prove the right part of~(15)~is nonnegative.~In fact
$$
\begin{array}{l}
  \quad \sum\limits_{\alpha  < \beta } {\left( {\gamma \left( {c_1^{\alpha \beta } } \right)^2  + \delta \left( {c_2^{\alpha \beta } } \right)^2 } \right)} \mu _\alpha  \mu _\beta  \\
  = \gamma \sum\limits_{2 \le \alpha  < \beta } {\left( {c_1^{\alpha \beta } } \right)^2 } \mu _\alpha  \mu _\beta   + \delta \sum\limits_{1 \le \alpha  < \beta } {\left( {c_2^{\alpha \beta } } \right)^2 } \mu _\alpha  \mu _\beta   \\
  = \gamma \sum\limits_{\beta  > 2} {\left( {c_1^{2\beta } } \right)^2 } \mu _2 \mu _\beta   + \delta \sum\limits_{\beta  > 2} {\left( {c_2^{1\beta } } \right)^2 } \mu _1 \mu _\beta   + \gamma \sum\limits_{3 \le \alpha  < \beta } {\left( {c_1^{\alpha \beta } } \right)^2 } \mu _\alpha  \mu
  _\beta \\
  \quad + \delta \sum\limits_{3 \le \alpha  < \beta } {\left( {c_2^{\alpha \beta } } \right)^2 } \mu _\alpha  \mu _\beta   \\
  = \sum\limits_{\beta  > 2} {\left( {c_1^{2\beta } } \right)^2 } \left( {\gamma \mu _2  + \delta \mu _1 } \right)\mu _\beta   + \gamma \sum\limits_{3 \le \alpha  < \beta } {\left( {c_1^{\alpha \beta } } \right)^2 } \mu _\alpha  \mu _\beta \\
  \quad + \delta \sum\limits_{3 \le \alpha  < \beta } {\left( {c_2^{\alpha \beta } } \right)^2 } \mu _\alpha  \mu _\beta   \\
  = \delta \sum\limits_{\beta  > 2} {\left( {c_1^{2\beta } } \right)^2 } \left( {\frac{\gamma }{\delta }\mu _2  + \mu _1 } \right)\mu _\beta   + \gamma \sum\limits_{3 \le \alpha  < \beta } {\left( {c_1^{\alpha \beta } } \right)^2 } \mu _\alpha  \mu _\beta \\
  \quad + \delta \sum\limits_{3 \le \alpha  < \beta } {\left( {c_2^{\alpha \beta } } \right)^2 } \mu _\alpha  \mu _\beta   \\
 \end{array}
$$
Since~$ \gamma  = \lambda _1  \ge 1 - \left( {\lambda _1  + \lambda
_2 } \right)\lambda _2  = \delta  > 0$,~we have
$$
\begin{array}{l}
 \quad \sum\limits_{\alpha  < \beta } {\left( {\gamma \left( {c_1^{\alpha \beta } } \right)^2  + \delta \left( {c_2^{\alpha \beta } } \right)^2 } \right)} \mu _\alpha  \mu _\beta   \\
 \ge \delta \sum\limits_{\beta  > 2} {\left( {c_1^{2\beta } } \right)^2 } \left( {\frac{\delta }{\gamma }\mu _2  + \mu _1 } \right)\mu _\beta   + \gamma \sum\limits_{3 \le \alpha  < \beta } {\left( {c_1^{\alpha \beta } } \right)^2 } \mu _\alpha  \mu _\beta \\
 \quad + \delta \sum\limits_{3 \le \alpha  < \beta } {\left( {c_2^{\alpha \beta } } \right)^2 } \mu _\alpha  \mu _\beta \\
 \ge 0. \\
 \end{array}
$$
\end{proof}
\end{proof}
Now the Theorem 1.3 follows from Lemma 3.1 and 3.2 easily.
\begin{proof}[Proof of Theorem 1.3]
By using Lemma 3.1 and 3.2,~we can derive directly the weak maximum
principle for~$ \left( {\lambda _1 ,\lambda _2 } \right)
$-nonnegativity,~which completes the proof of Theorem 1.3.
\end{proof}

\section{Proof of the Strong Maximum Principle for~$ \left( {\lambda _1 ,\lambda _2 } \right)
$-Nonnegativity}

\begin{proof}[Proof of Theorem 1.4]
First by Theorem 1.3 we have that~$ \mathcal {R}(g(t)) $~is also $
\left( {\lambda _1 ,\lambda _2 } \right) $-nonnegative for
all~$t>0$.~We will prove~(ii)~firstly.
\\

\noindent \emph{Proof of}~(ii).~As the discussion in the proof of
Theorem 1.3,~we only need to consider the case that~$ \left( {\alpha
,\beta ,\gamma } \right) = \left( {1,2,3} \right)$.~Firstly we prove
the result for~$ C_{1,\lambda _1 ,\lambda _2 ,\alpha ,\beta ,\gamma
} \left( \mathcal {R} \right) $.~Given a smooth nonnegative
function~$ \varphi \left( x \right)$~satisfying
$$
\varphi \left( x \right) \le \frac{{\mu _1 \left( {\mathcal
{R}\left( {x,0} \right)} \right) + \lambda _1 \mu _2 \left(
{\mathcal {R}\left( {x,0} \right)} \right) + \lambda _2 \mu _3
\left( {\mathcal {R}\left( {x,0} \right)} \right)}}{{1 + \lambda _1
+ \lambda _2 }}
$$
for all~$ x \in M^n$.~We also assume that there exists~$ x_0 \in
M^n$~such that
$$
\varphi \left( {x_0 } \right) \ge \frac{{\mu _1 \left( {\mathcal
{R}\left( {x_0 ,0} \right)} \right) + \lambda _1 \mu _2 \left(
{\mathcal {R}\left( {x_0 ,0} \right)} \right) + \lambda _2 \mu _3
\left( {\mathcal {R}\left( {x_0 ,0} \right)} \right)}}{{2\left( {1 +
\lambda _1 + \lambda _2 } \right)}}
$$

Let~$f (x,t)$~be a solution to
$$
 {\frac{{\partial f}}{{\partial t}} =
\Delta f - Af}
$$
such that~$ {f\left( {x,0} \right) = \varphi \left( x \right)}$,~and
define
$$
 {\tilde{\mathcal {R}}} \left( {x,t} \right) = \mathcal
{R}\left( {x,t} \right) + \left( {\varepsilon e^{At}  - f\left(
{x,t} \right)} \right)id_{\Lambda ^2 } \left( x \right)
$$
where~$\varepsilon
> 0$.~Then for \emph{A} sufficiently large and by the Ricci flow equation,~we
can prove that~(see [6])
$$
\frac{\partial }{{\partial t}}{\tilde{\mathcal {R}}}
> \Delta {\tilde{\mathcal {R}}} + {\tilde{\mathcal {R}}}
^2  + {\tilde{\mathcal {R}}}^\# \eqno(16)
$$
for~$ \varepsilon  \in \left( {0,e^{ - AT} } \right]$.~Moreover,
when~$t=0$,~if follows that
$$
{\tilde{\mathcal {R}}} \left( {x,0} \right) = \mathcal {R}\left(
{x,0} \right) + \left( {\varepsilon  - f\left( {x,0} \right)}
\right)id_{\Lambda ^2 } \left( x \right) = \mathcal {R}\left( {x,0}
\right) + \left( {\varepsilon - \varphi \left( x \right)}
\right)id_{\Lambda ^2 } \left( x \right).
$$
Then by using the definition of~$ \varphi \left( x \right) $~we have
$$
\begin{array}{l}
 \quad \left( {\tilde{\mathcal {R}}\left( {x,0} \right)} \right)_{11}  + \lambda _1 \left( {\tilde{\mathcal {R}}\left( {x,0} \right)} \right)_{22}  + \lambda _2 \left( {\tilde{\mathcal {R}}\left( {x,0} \right)} \right)_{33}  \\
\\
  = \mu _1 \left( {\mathcal {R}\left( {x,0} \right)} \right) + \lambda _1 \mu _2 \left( {\mathcal {R}\left( {x,0} \right)} \right) + \lambda _2 \mu _3 \left( {\mathcal {R}\left( {x,0} \right)} \right) - \left( {1 + \lambda _1  + \lambda _2 } \right)\varphi \left( x \right) \\
 \quad + \left( {1 + \lambda _1  + \lambda _2 } \right)\varepsilon
 \\
 > 0, \\
 \end{array}
$$
By using the maximum principle Lemma 2.1 to~(16),~we can get
$$
\left( {\tilde{\mathcal {R}}\left( {x,t} \right)} \right)_{11}  +
\lambda _1 \left( {\tilde{\mathcal {R}}\left( {x,t} \right)}
\right)_{22}  + \lambda _2 \left( {\tilde{\mathcal {R}}\left( {x,t}
\right)} \right)_{33}  \ge 0
$$
for all~$t>0$.~Then taking the limit as~$ \varepsilon  \to 0$,~we
conclude that
$$
\begin{array}{l}
 \mu _1 \left( {\mathcal {R}\left( {x,t} \right)} \right) + \lambda _1 \mu _2 \left( {\mathcal {R}\left( {x,t} \right)} \right) + \lambda _2 \mu _3 \left( {\mathcal {R}\left( {x,t} \right)} \right) \\
  - \left( {1 + \lambda _1  + \lambda _2 } \right)f\left( {x,t} \right)id_{\Lambda ^2 } \left( x \right) \ge 0 \\
 \end{array}
$$
for arbitrary~$ \left( {x,t} \right) \in M^n  \times \left[ {0,T}
\right)$.

On the other hand,~since~$f (x,t)$~is a solution to the parabolic
equation
$$
 {\frac{{\partial f}}{{\partial t}} = \Delta f - Af}
$$
such that~$ {f\left( {x_0,0} \right) = \varphi \left( x_0 \right)}>
0 $,~by the strong maximum principle for the parabolic equation we
have~$f (x,t)>0$~for arbitrary~$ \left( {x,t} \right) \in M^n \times
\left( {0,T} \right)$.~Hence
$$
\mu _1 \left( {\mathcal {R}\left( {x,t} \right)} \right) + \lambda
_1 \mu _2 \left( {\mathcal {R}\left( {x,t} \right)} \right) +
\lambda _2 \mu _3 \left( {\mathcal {R}\left( {x,t} \right)} \right)
> 0
$$
for arbitrary~$ \left( {x,t} \right) \in M^n  \times \left( {0,T}
\right)$.

To prove the similar result for~$ C_{2,\lambda _1 ,\lambda _2
,\alpha ,\beta } \left( \mathcal {R} \right) $,~we only need to
select a smooth nonnegative function~$ \varphi' \left( x
\right)$~satisfying
$$
\varphi' \left( x \right) \le \frac{{\lambda _1 \mu _1 \left(
{\mathcal {R}\left( {x,0} \right)} \right) + \left( {1 - \left(
{\lambda _1  + \lambda _2 } \right)\lambda _2 } \right)\mu _2 \left(
{\mathcal {R}\left( {x,0} \right)} \right)}}{{\lambda _1  + \left(
{1 - \left( {\lambda _1  + \lambda _2 } \right)\lambda _2 }
\right)}}
$$
for all~$ x \in M^n$~and there exists~$ x_1 \in M^n$~such that
$$
\varphi' \left( {x_1 } \right) \ge \frac{{\lambda _1 \mu _1 \left(
{\mathcal {R}\left( {x_1 ,0} \right)} \right) + \left( {1 - \left(
{\lambda _1 + \lambda _2 } \right)\lambda _2 } \right)\mu _2 \left(
{\mathcal {R}\left( {x_1 ,0} \right)} \right)}}{{2\left( {\lambda _1
+ \left( {1 - \left( {\lambda _1  + \lambda _2 } \right)\lambda _2 }
\right)} \right)}}.
$$
We can also select a solution to
$$
 {\frac{{\partial f'}}{{\partial t}} =
\Delta f' - A' f'}
$$
$f'(x,t)$~such that~$ {f'\left( {x,0} \right) = \varphi' \left( x
\right)}$,~and define
$$
 {\tilde{\mathcal {R}}}' \left( {x,t} \right) = \mathcal
{R}'\left( {x,t} \right) + \left( {\varepsilon' e^{A't}  - f'\left(
{x,t} \right)} \right)id_{\Lambda ^2 } \left( x \right)
$$
where~$\varepsilon' > 0$.~Then the left proof is similar to the one
for $ C_{1,\lambda _1 ,\lambda _2 ,\alpha ,\beta ,\gamma } \left(
\mathcal {R} \right) $.
\\

\noindent \emph{Proof of}~(i).~If~$g(0)$~is~$ \left( {\lambda _1
,\lambda _2 } \right) $-nonnegative everywhere in~$ M^n $~and $
\left( {\lambda _1 ,\lambda _2 } \right) $-positive at a point in~$
M^n $,~then by~(ii),~$g(t)$~is~$ \left( {\lambda _1 ,\lambda _2 }
\right) $-positive everywhere for~$ 0 < t < T $.~Otherwise without
loss of generality we only need to consider the case that for some~$
0 < t_0 < T $~such that
$$
\mu _1 \left( {\mathcal {R}\left( {x_0 ,t_0 } \right)} \right) +
\lambda _1 \mu _2 \left( {\mathcal {R}\left( {x_0 ,t_0 } \right)}
\right) + \lambda _2 \mu _3 \left( {\mathcal {R}\left( {x_0 ,t_0 }
\right)} \right) = 0 \eqno(17)
$$
or
$$
\lambda _1 \mu _1 \left( {\mathcal {R}\left( {x_0 ,t_0 } \right)}
\right) + \left( {1 - \left( {\lambda _1  + \lambda _2 }
\right)\lambda _2 } \right)\mu _2 \left( {\mathcal {R}\left( {x_0
,t_0 } \right)} \right) = 0 \eqno(18)
$$
at point~$x_0$,~then we consider these two cases:
\\

\noindent If~(17)~satisfies,~by~(ii),~we have
$$
\mu _1 \left( {\mathcal {R}\left( {x,t} \right)} \right) + \lambda
_1 \mu _2 \left( {\mathcal {R}\left( {x,t} \right)} \right) +
\lambda _2 \mu _3 \left( {\mathcal {R}\left( {x,t} \right)} \right)
= 0
$$
for arbitrary~$ \left( {x,t} \right) \in M^n  \times \left[ {0,t_0 }
\right]$.~We will prove the following result:
$$
\mu _1 \left( {\mathcal {R}\left( {x,t} \right)} \right) = \mu _2
\left( {\mathcal {R}\left( {x,t} \right)} \right) = \mu _3 \left(
{\mathcal {R}\left( {x,t} \right)} \right) = 0 \eqno(19)
$$
for arbitrary~$ \left( {x,t} \right) \in M^n  \times \left[ {0,t_0 }
\right]$.

To prove~(19),~we consider any~$ \left( {x,t} \right) \in M^n \times
\left( {0,t_0 } \right]$.~Let~$ \omega _1 ,\omega _2 $~and~$\omega
_3$

\noindent be unit 2-forms at~$ \left( {x ,t } \right)$,~which are
eigenvectors for~${\mathcal {R}}\left( {x ,t }
\right)$~corresponding~$ \mu _1 \left( {\mathcal {R}}\left( {x ,t }
\right) \right),\mu _2 \left( {\mathcal {R}}\left( {x ,t } \right)
\right)$~and~$\mu _3 \left( {\mathcal {R}}\left( {x ,t } \right)
\right)$~respectively.~Parallel translate~$ \omega _1 ,\omega _2
,$~and~$\omega _3$~along geodesics emanating from~$x$~with respect
to~$g(t)$~to define~$ \omega _1 ,\omega _2 $~and $\omega _3$~in a
space-time neighborhood of~$ \left( {x ,t } \right)$,~where~$ \omega
_1 ,\omega _2 $~and~$\omega _3$ are independent of time (see [8])
.~By matrix analysis (see [8]) we can obtain that for arbitrary~$
\left( {x,t} \right) \in M^n$:
$$
\begin{array}{l}
 \quad \mathcal {R}\left( {\omega _1 ,\omega _1 } \right) + \lambda _1 \mathcal {R}\left( {\omega _2 ,\omega _2 } \right) + \lambda _2 \mathcal {R}\left( {\omega _3 ,\omega _3 } \right) \\
  = \inf \left\{ \begin{array}{l}
 \mathcal {R}\left( {V_i ,V_i } \right) + \lambda _1 \mathcal {R}\left( {V_j ,V_j } \right) + \lambda _2 \mathcal {R}\left( {V_k ,V_k } \right)\left| {V_i  \bot V_j  \bot V_k,} \right. \\
 \left\| {V_i } \right\| = \left\| {V_j } \right\| = \left\| {V_k } \right\| = 1,1 \le i,j,k \le N \\
 \end{array} \right\}, \\
 \end{array} \eqno(20)
$$
where~$ \left\|  \cdot  \right\|$~denotes the metric for the space~$
\Lambda ^2 TM^n $.~Since the curvature operator~$\mathcal {R}$~is
actually a linear operator of~$ \Lambda ^2 TM^n $,~by using
Uhlenbeck trick,~we can prove the metric is invariant along the
Ricci flow~(see [7] for details).~This implies that~$ \omega _1
,\omega _2 $~and~$\omega _3$~are also unit vectors in the space-time
neighborhood of~$ \left( {x ,t } \right)$,~then by~(20)~we have
$$
\begin{array}{l}
 \quad  \left( {\mathcal {R}\left( {\omega _1 ,\omega _1 } \right) + \lambda _1 \mathcal {R}\left( {\omega _2 ,\omega _2 } \right) + \lambda _2 \mathcal {R}\left( {\omega _3 ,\omega _3 } \right)} \right)\left( {x',t'} \right) \\
\\
  \ge \mu _1 \left( {\mathcal {R}\left( {x',t'} \right)} \right) + \lambda _1 \mu _2 \left( {\mathcal {R}\left( {x',t'} \right)} \right) + \lambda _2 \mu _3 \left( {\mathcal {R}\left( {x',t'} \right)} \right) \\
  = 0 \\
 \end{array}
$$
for arbitrary~$ \left( {x',t'} \right) \in M^n  \times \left[ {0,t_0
} \right]$.~On the other hand,~for arbitrary~$ \left( {x,t} \right)
\in M^n \times \left[ {0,t_0 } \right]$,~we have
$$
\begin{array}{l}
 \quad  \left( {\mathcal {R}\left( {\omega _1 ,\omega _1 } \right) + \lambda _1 \mathcal {R}\left( {\omega _2 ,\omega _2 } \right) + \lambda _2 \mathcal {R}\left( {\omega _3 ,\omega _3 } \right)} \right)\left( {x,t} \right) \\
\\
  = \mu _1 \left( {\mathcal {R}\left( {x,t} \right)} \right) + \lambda _1 \mu _2 \left( {\mathcal {R}\left( {x,t} \right)} \right) + \lambda _2 \mu _3 \left( {\mathcal {R}\left( {x,t} \right)} \right) \\
  = 0 \\
 \end{array}
$$
then it follows that at~$ \left( {x ,t } \right)$:
\[
\begin{split}
 0& \ge \frac{\partial }{{\partial t}}\left( {\mathcal {R}\left( {\omega _1 ,\omega _1 } \right) + \lambda _1 \mathcal {R}\left( {\omega _2 ,\omega _2 } \right) + \lambda _2 \mathcal {R}\left( {\omega _3 ,\omega _3 } \right)} \right) \\
  & = \left( {\frac{\partial }{{\partial t}}\mathcal {R}} \right)\left( {\omega _1 ,\omega _1 } \right) + \lambda _1 \left( {\frac{\partial }{{\partial t}}\mathcal {R}} \right)\left( {\omega _2 ,\omega _2 } \right) + \lambda _2 \left( {\frac{\partial }{{\partial t}}\mathcal {R}} \right)\left( {\omega _3 ,\omega _3 } \right) \\
  & = \left( {\Delta \mathcal {R} + \mathcal {R}^2  + \mathcal {R}^\#  } \right)\left( {\omega _1 ,\omega _1 } \right) + \lambda _1 \left( {\Delta \mathcal {R} + \mathcal {R}^2  + \mathcal {R}^\#  } \right)\left( {\omega _2 ,\omega _2 } \right) \\
  & \quad +\lambda _2 \left( {\Delta \mathcal {R} + \mathcal {R}^2  + \mathcal {R}^\#  } \right)\left( {\omega _3 ,\omega _3 } \right) \\
  & = \Delta \left( {\mathcal {R}\left( {\omega _1 ,\omega _1 } \right) + \lambda _1 \mathcal {R}\left( {\omega _2 ,\omega _2 } \right) + \lambda _2 \mathcal {R}\left( {\omega _3 ,\omega _3 } \right)} \right) + \mu _1 \left( \mathcal {R} \right)^2  + \lambda _1 \mu _2 \left( \mathcal {R} \right)^2  \\
  & \quad + \lambda _2 \mu _3 \left( \mathcal {R} \right)^2  + \mathcal {R}^\#  \left( {\omega _1 ,\omega _1 } \right) + \lambda _1 \mathcal {R}^\#  \left( {\omega _2 ,\omega _2 } \right) + \lambda _2 \mathcal {R}^\#  \left( {\omega _3 ,\omega _3 } \right) \\
  & \ge \mu _1 \left( \mathcal {R} \right)^2  + \lambda _1 \mu _2 \left( \mathcal {R} \right)^2  + \lambda _2 \mu _3 \left( \mathcal {R} \right)^2 , \\
 \end{split}
\]
where to obtain the last inequality we used
$$
\sum\limits_{\alpha  < \beta } {\left( {\left( {c_1^{\alpha \beta }
} \right)^2  + \lambda _1 \left( {c_2^{\alpha \beta } } \right)^2  +
\lambda _2 \left( {c_3^{\alpha \beta } } \right)^2 } \right)} \mu
_\alpha  \mu _\beta \ge 0,
$$
and the fact by~(20)~that
$$
 \mathcal {R}\left( {\omega _1 ,\omega _1 }
\right) + \lambda _1 \mathcal {R}\left( {\omega _2 ,\omega _2 }
\right) + \lambda _2 \mathcal {R}\left( {\omega _3 ,\omega _3 }
\right) \ge 0
$$
for arbitrary~$ x' \ne x $,~while at~$ \left( {x ,t } \right)$~is
0.~Hence
$$
\mu _1 \left( {\mathcal {R}\left( {x,t} \right)} \right) = \mu _2
\left( {\mathcal {R}\left( {x,t} \right)} \right) = \mu _3 \left(
{\mathcal {R}\left( {x,t} \right)} \right)  = 0
$$
for arbitrary~$ \left( {x,t} \right) \in M^n  \times \left[ {0,t_0
}\right]$,~which implies that~$ \mathcal {R}\left( {x,t} \right) \ge
0$~for arbitrary~$ \left( {x,t} \right) \in M^n \times \left[ {0,t_0
} \right]$.~Since~$ \mathcal {R}^2  + \mathcal {R}^\# $~satisfies
the null-eigenvector assumption,~then by Theorem 2.1,~$ \mathcal
{R}\left( {x,t} \right) \ge 0$~is preserved under the Ricci
flow,~thus we have~$ \mathcal {R}\left( {x,t} \right) \ge 0$~for
arbitrary~$ \left( {x,t} \right) \in M^n \times \left[ {0,T } \right
)$.
\\

If~(18)~satisfies,~by~(ii),~we also have
$$
\lambda _1 \mu _1 \left( {\mathcal {R}\left( {x,t} \right)} \right)
+ \left( {1 - \left( {\lambda _1  + \lambda _2 } \right)\lambda _2 }
\right)\mu _2 \left( {\mathcal {R}\left( {x,t} \right)} \right) = 0
$$
for arbitrary~$ \left( {x,t} \right) \in M^n  \times \left[ {0,t_0 }
\right]$.~As the proof of~(19)~under the condition~(17),~we consider
any~$ \left( {x,t} \right) \in M^n \times \left( {0,t_0 }
\right]$.~Let~$ \omega _1 $~and~$\omega _2$~be unit 2-forms at~$
\left( {x ,t } \right)$,~which are eigenvectors for~${\mathcal
{R}}\left( {x ,t } \right)$~corresponding~$ \mu _1 \left( {\mathcal
{R}}\left( {x ,t } \right) \right)$~and $\mu _2 \left( {\mathcal
{R}}\left( {x ,t } \right) \right)$ respectively.~Parallel
translate~$ \omega _1 $~and~$\omega _2$

\noindent along geodesics emanating from~$x$~with respect
to~$g(t)$~to define~$ \omega _1 $~and~$\omega _2 $

\noindent in a space-time neighborhood of~$ \left( {x ,t }
\right)$,~where~$ \omega _1 $~and~$\omega _2$~are independent of
time.~By matrix analysis we also have
$$
\begin{array}{l}
 \quad \lambda _1 \mathcal {R}\left( {\omega _1 ,\omega _1 } \right) + \left( {1 - \left( {\lambda _1  + \lambda _2 } \right)\lambda _2 } \right)\mathcal {R}\left( {\omega _2 ,\omega _2 } \right) \\
  = \inf \left\{ {\begin{array}{*{20}c}
   {\lambda _1 \mathcal {R}\left( {V_i ,V_i } \right) + \left( {1 - \left( {\lambda _1  + \lambda _2 } \right)\lambda _2 } \right)\mathcal {R}\left( {V_j ,V_j } \right)\left| {V_i  \bot V_j ,} \right.}  \\
   {\left\| {V_i } \right\| = \left\| {V_j } \right\| = 1,1 \le i,j \le N}  \\
\end{array}} \right\}, \\
 \end{array} \eqno(21)
$$
where~$ \left\|  \cdot  \right\|$~also denotes the metric for the
space~$ \Lambda ^2 TM^n $.~By the same reason after~(20)~we also
have~$ \omega _1 $~and~$ \omega _2 $~are also unit vectors in the
space-time neighborhood of~$ \left( {x ,t } \right)$,~then
by~(21)~we have
$$
\begin{array}{l}
  \quad \left( {\lambda _1 \mathcal {R}\left( {\omega _1 ,\omega _1 } \right) + \left( {1 - \left( {\lambda _1  + \lambda _2 } \right)\lambda _2 } \right)\mathcal {R}\left( {\omega _2 ,\omega _2 } \right)} \right)\left( {x',t'} \right) \\
\\
  \ge \lambda _1 \mu _1 \left( {\mathcal {R}\left( {x',t'} \right)} \right) + \left( {1 - \left( {\lambda _1  + \lambda _2 } \right)\lambda _2 } \right)\mu _2 \left( {\mathcal {R}\left( {x',t'} \right)} \right) \\
  = 0 \\
 \end{array}
$$
for arbitrary~$ \left( {x',t'} \right) \in M^n  \times \left[ {0,t_0
} \right]$.~~On the other hand,~for arbitrary~$ \left( {x,t} \right)
\in M^n \times \left[ {0,t_0 } \right]$,~we have
$$
\begin{array}{l}
  \quad \left( {\lambda _1 \mathcal {R}\left( {\omega _1 ,\omega _1 } \right) + \left( {1 - \left( {\lambda _1  + \lambda _2 } \right)\lambda _2 } \right)\mathcal {R}\left( {\omega _2 ,\omega _2 } \right)} \right)\left( {x,t} \right) \\
\\
  = \lambda _1 \mu _1 \left( {\mathcal {R}\left( {x,t} \right)} \right) + \left( {1 - \left( {\lambda _1  + \lambda _2 } \right)\lambda _2 } \right)\mu _2 \left( {\mathcal {R}\left( {x,t} \right)} \right) \\
  = 0 \\
 \end{array}
$$
then it follows that at~$ \left( {x ,t } \right)$:
\[
\begin{split}
 0& \ge \lambda _1 \left( {\frac{\partial }{{\partial t}}\mathcal {R}} \right)\left( {\omega _1 ,\omega _1 } \right) + \left( {1 - \left( {\lambda _1  + \lambda _2 } \right)\lambda _2 } \right)\left( {\frac{\partial }{{\partial t}}\mathcal {R}} \right)\left( {\omega _2 ,\omega _2 } \right) \\
  & = \lambda _1 \left( {\Delta \mathcal {R} + \mathcal {R}^2  + \mathcal {R}^\#  } \right)\left( {\omega _1 ,\omega _1 } \right) \\
  & \quad + \left( {1 - \left( {\lambda _1  + \lambda _2 } \right)\lambda _2 } \right)\left( {\Delta \mathcal {R} + \mathcal {R}^2  + \mathcal {R}^\#  } \right)\left( {\omega _2 ,\omega _2 } \right) \\
  & = \Delta \left( {\lambda _1 \mathcal {R}\left( {\omega _1 ,\omega _1 } \right) + \left( {1 - \left( {\lambda _1  + \lambda _2 } \right)\lambda _2 } \right)\mathcal {R}\left( {\omega _2 ,\omega _2 } \right)} \right) + \lambda _1 \mu _1 \left( \mathcal {R} \right)^2  \\
  & \quad + \left( {1 - \left( {\lambda _1  + \lambda _2 } \right)\lambda _2 } \right)\mu _2 \left( \mathcal {R} \right)^2  + \lambda _1 \mathcal {R}^\#  \left( {\omega _1 ,\omega _1 } \right) \\
  & \quad + \left( {1 - \left( {\lambda _1  + \lambda _2 } \right)\lambda _2 } \right)\mathcal {R}^\#  \left( {\omega _2 ,\omega _2 } \right) \\
  & \ge \lambda _1 \mu _1 \left( \mathcal {R} \right)^2  + \left( {1 - \left( {\lambda _1  + \lambda _2 } \right)\lambda _2 } \right)\mu _2 \left( \mathcal {R} \right)^2,  \\
 \end{split}
\]
where to obtain the last inequality we used
$$
\sum\limits_{\alpha  < \beta } {\left( {\lambda _1 \left(
{c_1^{\alpha \beta } } \right)^2  + \left( {1 - \left( {\lambda _1 +
\lambda _2 } \right)\lambda _2 } \right)\left( {c_2^{\alpha \beta }
} \right)^2 } \right)} \mu _\alpha  \mu _\beta   \ge 0,
$$
and the fact by~(21)~that
$$
 \lambda _1 \mathcal {R}\left( {\omega _1
,\omega _1 } \right) + \left( {1 - \left( {\lambda _1  + \lambda _2
} \right)\lambda _2 } \right)\mathcal {R}\left( {\omega _2 ,\omega
_2 } \right) \ge 0 $$ for arbitrary~$ x' \ne x $,~while at~$ \left(
{x ,t } \right)$~is 0.~Hence we also derive~(19).~As the proof
of~(19)~under the condition~(17),~we obtain~$ \mathcal {R}\left(
{x,t} \right) \ge 0$~for arbitrary~$ \left( {x,t} \right) \in M^n
\times \left[ {0,T } \right )$.
\end{proof}
Let~$ \left\{ {\omega _\alpha  } \right\}_{\alpha  = 1}^N  = \left\{
{e_i \Lambda e_j } \right\}_{i < j}$~be an orthonormal basis for~$
{so\left( n \right)}$,~where each~$ \alpha $~corresponds to a
pair~$(i,j)$~with~$i < j$.~We then present an interesting property
of~$ \left( {\lambda _1 ,\lambda _2 } \right) $-nonnegative
curvature operators~$ \mathcal {R} $~as following:
\begin{theorem}
The manifolds with~$ \left( {\lambda _1 ,\lambda _2 } \right)
$-nonnegativity have nonnegative scale curvature~$ Scal\left(
\mathcal {R} \right) $,~and with equality if and only if~$\mathcal
{R} = 0$.
\end{theorem}
\begin{proof}
We compute
\[
\begin{split}
 Tr\left( \mathcal {R} \right) & =  \sum\limits_{\alpha  = 1}^N {\left\langle {\mathcal {R}\left( {\omega _\alpha  } \right),\omega _\alpha  } \right\rangle } \\
 & = \sum\limits_{i < j} {\left\langle {\mathcal {R}\left( {e_i \Lambda e_j } \right),e_i \Lambda e_j } \right\rangle }  \\
 & = \frac{1}{2}\sum\limits_{i,j} {\mathcal {R}_{ijij} } \\
 & = \frac{1}{2}Tr\left( {Rc\left( \mathcal {R} \right)} \right) \\
 & = \frac{1}{2}Scal\left( \mathcal {R} \right). \\
 \end{split}
\]
On the other hand,~since~$ \mathcal {R} $~is~$ \left( {\lambda _1
,\lambda _2 } \right) $~nonnegative we have
\[
\begin{split}
 0 & \le \sum\limits_{\alpha  \ne \beta  \ne \gamma } {\mathcal {R}_{\alpha \alpha }  + \lambda _1 \mathcal {R}_{\beta \beta }  + \lambda _2 \mathcal {R}_{\gamma \gamma }  }  \\
   & = \left( {N - 1} \right)\left( {N - 2} \right)\left( {1 + \lambda _1  + \lambda _2  } \right)Tr\left( \mathcal {R} \right) \\
   & = \frac{1}{2}\left( {N - 1} \right)\left( {N - 2} \right)\left( {1 + \lambda _1  + \lambda _2  } \right)Scal\left( \mathcal {R} \right), \\
 \end{split}
\]
that is~$ Scal\left( \mathcal {R} \right) \ge 0$.~Hence if~$
Scal\left( \mathcal {R} \right) = 0$,~then it follows from~$ \left(
{\lambda _1 ,\lambda _2 } \right) $-nonnegative~$ \mathcal
{R}_{\alpha \alpha } + \lambda _1 \mathcal {R}_{\beta \beta } +
\lambda _2 \mathcal {R}_{\gamma \gamma } \ge 0$~that~$ \mathcal
{R}_{\alpha \alpha }  + \lambda _1 \mathcal {R}_{\beta \beta } +
\lambda _2 \mathcal {R}_{\gamma \gamma } = 0$,~for all~$ 1 \le
\alpha < \beta < \gamma \le N $.~Since~$ \mathcal {R}_{11} \le 0 \le
\mathcal {R}_{22} \le \cdots \le \mathcal {R}_{NN}$,~we have~$
\mathcal {R}_{\alpha \alpha }  = 0$~for any~$ 1 \le \alpha \le N $.
\end{proof}

\vspace{0.5em}

\noindent \textbf{Acknowledgment} We would especially like to
express our appreciation to professor Ben Chow for longtime
encouragement and meaningful discussions.

\end{document}